\documentclass{amsart}
\usepackage{latexsym, amssymb, amsmath}

\newtheorem{conj}{Conjecture}
\newtheorem{thm}{Theorem}[section]
\newtheorem{lem}[thm]{Lemma}
\newtheorem{cor}[thm]{Corollary}

\theoremstyle{definition}

\newcommand{\olapla}{\overline\Delta}
\newcommand{\onabla}{\overline\nabla}

\newcommand{\p}{\phi}

\title{Polyharmonic maps of order k\\ with finite $L^p$ k-energy
 into Euclidean spaces}

\author{Shun Maeta}

\thanks{}

\address{\footnotesize{Faculty of Tourism and Business Management Shumei University, Chiba 276-0003, Japan.}
 }
\footnotesize{
\email{shun.maeta@gmail.com~{\it or}~maeta@mailg.shumei-u.ac.jp}
}

\begin{document} 
\maketitle 
\markboth{polyharmonic maps of order k} 
{Shun Maeta} 

\begin{abstract} 
We consider polyharmonic maps $\phi:(M,g)\rightarrow \mathbb{E}^n$ of order $k$ from a complete Riemannian manifold into the Euclidean space and let $p$ be a real constant satisfying $1<p<\infty$.
$(i)$ If, 
$\int_M|W^{k-1}|^{p}dv_g<\infty,$
and
$\int_M|\onabla W^{k-2}|^2dv_g<\infty.$
 Then $\p$ is a polyharmonic map of order $k-1$.
$(ii)$ If, 
$\int_M|W^{k-1}|^{p}dv_g<\infty,$ 
and 
${\rm Vol}(M,g)=\infty.$
Then $\p$ is a polyharmonic map of order $k-1$.
Here, $W^s=\olapla^{s-1}\tau(\p)\ (s=1,2,\cdots)$ and $W^0=\p$.
As a corollary, we give an affirmative partial answer to generalized Chen's conjecture.
\end{abstract}

\qquad\\


\section{Introduction}\label{intro} 
The theory of harmonic maps has been applied into various fields in differential geometry.
 Harmonic maps between two Riemannian manifolds are
 critical points of the {\em energy} functional 
 $E(\p)=\frac{1}{2}\int_M|d\p|^2dv_g,$
  for smooth maps $\p:(M^m,g)\rightarrow (N^n,h)$ from an $m$-dimensional Riemannian manifold into an $n$-dimensional Riemannian manifold, 
 where $dv_g$ denotes the volume element of $g.$
The Euler-Lagrange equation of $E$ is $\tau(\p)={\rm Trace}\nabla d\p=0,$
where $\tau(\p)$ is called the {\em tension field} of $\p.$
 A map $\p:(M,g)\rightarrow (N,h)$ is called a {\em harmonic map} if $\tau(\p)=0.$

In 1983, J. Eells and L. Lemaire \cite{jell1} proposed the problem to consider polyharmonic maps of order $k$ ($k$-harmonic maps) which are critical points of the {\em $k$-energy} functional 
$E_k (\phi )=\frac{1}{2}\int_M |(d+d^*)^{k-2}\tau (\phi)| ^2 dv_g,$
 on the space of smooth maps between two Riemannian manifolds.
 Polyharmonic maps are, by definition, a generalization of harmonic maps.
 If $\p:(M,g)\rightarrow \mathbb{E}^n$ is a smooth map, then the Euler-Lagrange equation of $E_k$
 is 
 $$\olapla^{k-1}\tau(\p)=0,$$
  where $\olapla:=\displaystyle\sum^m_{i=1}\left(\onabla_{e_i}\onabla_{e_i}-\onabla_{\nabla_{e_i}e_i}\right)$, and $\onabla$ is the induced connection.

Polyharmonic maps of order $2$ are called biharmonic maps.
There are many studies for the biharmonic theory.
One of the most interesting problem in the biharmonic theory is Chen's conjecture. 
 In 1988, B. Y. Chen raised the following problem:
\vspace{5pt}

\begin{conj}
[\cite{byc1}]\label{Chen}
Any biharmonic submanifold in $\mathbb{E}^n$ is minimal. 
\end{conj}

\vspace{5pt}
Here, we say for a submanifold $M$ in $\mathbb{E}^n$, to be biharmonic if an isometric immersion $\p:(M,g)\rightarrow \mathbb{E}^n$ is biharmonic.
There are many affirmative partial answers to Chen's conjecture (cf. \cite{kasm1},\ \cite{Leuven},\ \cite{Dimi},\ \cite{Hasanis-Vlachos},\ \cite{sm10},\ \cite{nnhusg1},\ \cite{yl2013}, etc.)

 On the other hand, Chen's conjecture was generalized as follows (cf. \cite{sm2},\ see also \cite{byc2013}):
 \begin{conj}[\cite{sm2}]\label{generalized Chen}
 Any polyharmonic submanifold of order $k$ in $\mathbb{E}^n$ is minimal.
 \end{conj}
 
 The author showed that polyharmonic curves parametrized by arc length of order $k$ is a straight line (cf.\ \cite{sm2}).
Recently, N. Nakauchi and H. Urakawa gave an affirmative partial answer to Conjecture $\ref{generalized Chen}$ as follows.

\begin{thm}[\cite{nnhu2013}]\label{Th of N-U}
Let $\p:(M,g)\rightarrow \mathbb{E}^n$ be a polyharmonic map of order $k$, that is, $\olapla^{k-1} \tau(\p)=0$, from  a complete Riemannian manifold $(M,g)$ into $\mathbb{E}^n$.

$(i)$ If  
$$\int_M|W^{q-1}|^{2}dv_g<\infty,\ \ \ \ \ \text{for~all}~q=2, 3,\cdots, k,$$
and
$$\int_M|\onabla W^{q-1}|^2dv_g<\infty,\ \ \ \ \ \text{for~all}~q=1,2,\cdots, k-1.$$
Then $\p$ is harmonic.

$(ii)$ If  
$$\int_M|W^{q-1}|^{2}dv_g<\infty,\ \ \ \ \ \text{for~all}~q=2, 3,\cdots, k,$$
and 
$${\rm Vol}(M,g)=\infty.$$
Then $\p$ is harmonic.

Here, $W^s=\olapla^{s-1}\tau(\p)~ (s=1,2,\cdots)$ and $W^0=\p$.

\end{thm}

\vspace{5pt}

Our main result of this paper is the following.

\vspace{5pt}

\begin{thm}\label{main Th}
Let $\p:(M,g)\rightarrow \mathbb{E}^n$ be a polyharmonic map of order $k$, that is, $\olapla^{k-1} \tau(\p)=0$, from  a complete Riemannian manifold $(M,g)$ into $\mathbb{E}^n$ and let $p$ be a real constant satisfying $1<p<\infty$.

$(i)$ If
$$\int_M|W^{k-1}|^{p}dv_g<\infty,$$
and
$$\int_M|\onabla W^{k-2}|^2dv_g<\infty.$$
Then $\p$ is a polyharmonic map of order $k-1$.

$(ii)$ If
$$\int_M|W^{k-1}|^{p}dv_g<\infty,$$
and
$${\rm Vol}(M,g)=\infty.$$
Then $\p$ is a polyharmonic map of order $k-1$.

Here, $W^{s}=\olapla^{s-1}\tau(\p)$ $(s=1,2,\cdots)$ and $W^0=\p$.

\end{thm}

\vspace{5pt}

By Theorem $\ref{main Th}$, we obtain the following corollary.

\vspace{5pt}

\begin{cor}\label{main Cor}
Let $\p:(M,g)\rightarrow \mathbb{E}^n$ be a polyharmonic map of order $k$, that is, $\olapla^{k-1} \tau(\p)=0$, from  a complete Riemannian manifold $(M,g)$ into $\mathbb{E}^n$  and let $p$ be a real constant satisfying $1<p<\infty$.

$(i)$ If
$$\int_M|W^{q-1}|^{p}dv_g<\infty,\ \ \ \ \ \text{for~all}~q=2,3,\cdots, k,$$
and
$$\int_M|\onabla W^{q-1}|^2dv_g<\infty,\ \ \ \ \ \text{for~all}~q=1,2,\cdots,k-1.$$
Then $\p$ is harmonic.

$(ii)$ If
$$\int_M|W^{q-1}|^{p}dv_g<\infty,\ \ \ \ \ \text{for~all}~q=2,3,\cdots, k,$$
and
$${\rm Vol}(M,g)=\infty.$$
Then $\p$ is harmonic.

Here, $W^s=\olapla^{s-1}\tau(\p)~ (s=1,2,\cdots)$ and $W^0=\p$.

\end{cor}

 This corollary is a generalization of Theorem $\ref{Th of N-U}$
 and give an affirmative partial answer to generalized Chen's conjecture.

The remaining sections are organized as follows. 
Section~$\ref{Pre}$ contains some necessary definitions and preliminary geometric results.
 In section~$\ref{map}$, we show our main theorem.

\quad\\

\noindent
{\bf Acknowledgment\\}
The author would like to express his  gratitude to
 Professor Nobumitsu Nakauchi and Professor Hajime Urakawa for their useful advice.

\qquad\\

\section{Preliminaries}\label{Pre} 

In this section, we recall polyharmonic maps.

Let $(M,g)$ be an $m$-dimensional Riemannian manifold.
Assume that $\p:(M,g)\rightarrow \mathbb{E}^n$ be a smooth map.
We denote by $\nabla$ the Levi-Civita connection on $(M,g)$ and by $\onabla$ the induced connection.

\vspace{10pt}

Let us recall the definition of a harmonic map $\p:(M,g)\rightarrow \mathbb{E}^n$.
For a smooth map $\phi:(M,g)\rightarrow \mathbb{E}^n$, the {\em energy} of $\phi$ is defined by
$$E(\phi) =\frac{1}{2}\int_M|d\phi|^2 dv_g.$$
The Euler-Lagrange equation of $E$ is 
$$ \tau(\p)=\displaystyle \sum^m_{i=1}\{\onabla_{e_i}d\p(e_i)-d\p(\nabla_{e_i}e_i)\}=0,$$
where $ \tau(\p)$ is called the {\em tension field} of $\p$ and $\{e_i\}_{i=1}^m$ is an orthonormal frame field on $M$.
 A map $\p$ is called a {\em harmonic map} if $\tau(\p)=0$. 

\vspace{10pt}

In 1983, J. Eells and L. Lemaire \cite{jell1} proposed the problem to consider polyharmonic maps of order $k$ ($k$-harmonic maps) which are critical points of the {\em $k$-energy} functional 
$E_k (\phi )=\frac{1}{2}\int_M |(d+d^*)^{k-2}\tau (\phi)| ^2 dv_g,$
 on the space of smooth maps $\p:(M,g)\rightarrow \mathbb{E}^n$.
The Euler-Lagrange equation of $E_k$ is 
$$\olapla^{k-1} \tau(\p)=0.$$
For convenience we define as follows.
\begin{align*}
W^q&:=\olapla^{q-1}\tau(\p)=\underbrace{\olapla\cdots\olapla}_{q-1}\tau(\p),\\
W^0&:=\p.
\end{align*}

\qquad\\


\section{Proof of main theorem}\label{map} 
In this section, we shall give a proof of our main theorem (Theorem $\ref{main Th}$).
We first show the following lemma.
\vspace{5pt}

\begin{lem}\label{key lem 1}
Let $\p:(M,g)\rightarrow \mathbb{E}^n$ be a polyharmonic map of order $k$ from  a complete Riemannian manifold $(M,g)$ into $\mathbb{E}^n$.

Assume that $p$ satisfies $1<p<\infty$.
If such an $p$,
$$\int_M| W^{k-1}|^{p}dv_g<\infty,$$
then $\onabla_X W^{k-1}=0$ for any vector field $X$ on $M$. In particular, $| W^{k-1}|$ is constant.
Here, $W^s=\olapla^{s-1}\tau(\p)$ $(s=1,2,\cdots)$ and $W^0=\p$.
\end{lem}

\begin{proof}
For a fixed point $x_0\in M$, and for every $0<r<\infty,$
 we first take a cut off function $\lambda$ on $M$ satisfying that 
 \begin{equation}
\left\{
 \begin{aligned}
&0\leq\lambda(x)\leq1\ \ \ (x\in M),\\
&\lambda(x)=1\ \ \ \ \ \ \ \ \ (x\in B_r(x_0)),\\
&\lambda(x)=0\ \ \ \ \ \ \ \ \ (x\not\in B_{2r}(x_0)),\\
&|\nabla\lambda|\leq\frac{C}{r}\ \ \ \ \ \ \ (x\in M),\ \ \ \text{for some constant $C$ independent of $r$},
\end{aligned} 
\right.
\end{equation}
where $B_r(x_0)$ and  $B_{2r}(x_0)$ are the balls centered at a fixed point $x_0\in M$ with radius $r$ and $2r$ respectively (cf. \cite{ka1}).
Since $\olapla W^{k-1}=W^{k}=\olapla^{k-1}\tau(\p)=0$, we have
\begin{equation}\label{0}
\begin{aligned}
0=&\int_M\langle -\olapla W^{k-1},\lambda^2| W^{k-1}|^{p-2}W^{k-1}\rangle dv_g.
\end{aligned}
\end{equation}
By $(\ref{0})$, we have
\begin{equation}\label{branch}
\begin{aligned}
0=&\int_M\langle -\olapla W^{k-1},\lambda^2| W^{k-1}|^{p-2} W^{k-1}\rangle dv_g\\
=&\int_M\langle \onabla W^{k-1},\onabla(\lambda^2| W^{k-1}|^{p-2} W^{k-1})\rangle dv_g\\
=&\int_M\sum_{i=1}^m\langle \onabla_{e_i} W^{k-1},(e_i\lambda^2)| W^{k-1}|^{p-2} W^{k-1}
+\lambda^2e_i\{(| W^{k-1}|^2)^{\frac{p-2}{2}}\} W^{k-1}\\
&\hspace{180pt}+\lambda^2| W^{k-1}|^{p-2}\onabla_{e_i} W^{k-1}\rangle dv_g\\
=&\int_M\sum_{i=1}^m\langle \onabla_{e_i} W^{k-1}, 2\lambda(e_i\lambda)| W^{k-1}|^{p-2} W^{k-1}\rangle dv_g\\
&+\int_M\sum_{i=1}^m\langle \onabla_{e_i} W^{k-1},\lambda^2(p-2)| W^{k-1}|^{p-4}\langle\onabla_{e_i} W^{k-1}, W^{k-1}\rangle  W^{k-1}\rangle dv_g\\
&+\int_M\sum_{i=1}^m\langle \onabla_{e_i} W^{k-1}, \lambda^2 | W^{k-1}|^{p-2} \onabla_{e_i} W^{k-1}\rangle dv_g.
\end{aligned}
\end{equation}
(i) The case of $p\geq2$. From $(\ref{branch})$, we have
 \begin{equation}\label{i-1}
 \begin{aligned}
 0\geq 
 &\int_M\sum_{i=1}^m\langle \onabla_{e_i} W^{k-1}, 2\lambda(e_i\lambda)| W^{k-1}|^{p-2} W^{k-1}\rangle dv_g\\
&+\int_M\sum_{i=1}^m\langle \onabla_{e_i} W^{k-1}, \lambda^2 | W^{k-1}|^{p-2} \onabla_{e_i} W^{k-1}\rangle dv_g.
 \end{aligned}
 \end{equation}
We consider the first term of the right hand side of $(\ref{i-1})$.
\begin{equation}\label{Using Young} 
\begin{aligned}
 &-2\int_M\sum_{i=1}^m\langle \onabla_{e_i} W^{k-1}, \lambda(e_i\lambda)| W^{k-1}|^{p-2} W^{k-1}\rangle dv_g\\
=&-2\int_M\sum_{i=1}^m\langle (e_i\lambda) | W^{k-1}|^{\frac{p}{2}-1} W^{k-1},
\lambda | W^{k-1}|^{\frac{p}{2}-1} \onabla_{e_i} W^{k-1}\rangle dv_g\\
\leq&~2\int_M|\nabla\lambda|^2| W^{k-1}|^{p}dv_g\\
&+\frac{1}{2}\int_M\lambda^2|W^{k-1}|^{p-2}|\onabla_{e_i}W^{k-1}|^2dv_g,
\end{aligned}
\end{equation}
where the inequality of $(\ref{Using Young})$, follows from the following inequality
\begin{equation}\label{Young's ineq.}
\pm2\langle V,U \rangle \leq \varepsilon |V|^2+\frac{1}{\varepsilon}|U|^2,\ \ \ \ \ \text{for all positive}\  \varepsilon >0,
\end{equation}
because of the inequality $0\leq |\sqrt{\varepsilon}V\pm \frac{1}{\sqrt{\varepsilon}}U|^2$.
The inequality $(\ref{Young's ineq.})$ is called {\em Young's inequality}.
Substituting $(\ref{Using Young})$ into $(\ref{i-1})$, we have
\begin{equation}\label{ast20}
\begin{aligned}
\int_M\lambda^2|W^{k-1}|^{p-2}|\onabla_{e_i}W^{k-1}|^2dv_g
\leq&~4\int_M|\nabla\lambda|^2 | W^{k-1}|^{p}dv_g\\
\leq&\int_M \frac{4C^2}{r^2}| W^{k-1}|^{p}dv_g.
\end{aligned}
\end{equation}
 By the assumption $\int_M| W^{k-1}|^{p}dv_g<\infty$,
  the right hand side of $(\ref{ast20})$ goes to zero and the left hand side of $(\ref{ast20})$ goes to 
  $$\int_M|W^{k-1}|^{p-2}|\onabla_{e_i}W^{k-1}|^2dv_g,$$
   since $\lambda=1$ on $B_r(x_0)$.
 Thus, we have
  $$\int_M|W^{k-1}|^{p-2}|\onabla_{e_i}W^{k-1}|^2dv_g=0.$$
   From this, we obtain for any vector field $X$ on $M$,
  \begin{equation}\label{eq. for constant0}
  \onabla_X W^{k-1}=0.
  \end{equation}
 By $(\ref{eq. for constant0})$,
  $$X| W^{k-1}|^2=2\langle \onabla_X W^{k-1}, W^{k-1}\rangle=0.$$
  Therefore we obtain $| W^{k-1}|$ is constant. 

(ii) The case of $p<2$. From $(\ref{branch})$, we have
\begin{equation}\label{ast}
\begin{aligned}
&\int_M\sum_{i=1}^m \lambda^2| W^{k-1}|^{p-2}|\onabla_{e_i} W^{k-1}|^2 dv_g\\
&=-2\int_M\sum_{i=1}^m\langle \onabla_{e_i} W^{k-1}, \lambda (e_i\lambda) | W^{k-1}|^{p-2} W^{k-1}\rangle dv_g\\
&\ \ \ -(p-2)\int_M\sum_{i=1}^m\langle \onabla_{e_i} W^{k-1},\lambda^2| W^{k-1}|^{p-4}\langle \onabla _{e_i} W^{k-1}, W^{k-1} \rangle  W^{k-1} \rangle dv_g.
\end{aligned}
\end{equation}

We shall consider the first term of the right hand side of $(\ref{ast})$.

\begin{equation}\label{1}
\begin{aligned}
-2&\int_M\sum_{i=1}^m\langle \onabla_{e_i} W^{k-1},\lambda (e_i \lambda)| W^{k-1}|^{p-2} W^{k-1}\rangle dv_g\\
=&-2\int_M\sum_{i=1}^m\langle \lambda | W^{k-1}|^{\frac{p}{2}-1} \onabla_{e_i} W^{k-1},
(e_i \lambda ) | W^{k-1}|^{\frac{p}{2}-1} W^{k-1}\rangle dv_g\\
\leq&~\varepsilon \int_M \sum_{i=1}^m \lambda^2 | W^{k-1}|^{p-2} |\onabla_{e_i} W^{k-1}|^2dv_g\\
&+\frac{1}{\varepsilon}\int_M(\nabla\lambda)^2| W^{k-1}|^{p}dv_g,
\end{aligned}
\end{equation} 
where the inequality of $(\ref{1})$ follows from Young's inequality, that is,
$$\pm2 \langle V, U\rangle \leq \varepsilon |V|^2+\frac{1}{\varepsilon}|U|^2,$$
for all positive $\varepsilon >0$,
 where we take $V= \lambda | W^{k-1}|^{\frac{p}{2}-1} \onabla_{e_i} W^{k-1} $ and $U=(e_i \lambda ) | W^{k-1}|^{\frac{p}{2}-1} W^{k-1}$. 

 We shall consider the second term of the right hand side of $(\ref{ast})$.

\begin{equation}\label{2}
\begin{aligned}
-(p-2)&\int_M\sum_{i=1}^m\langle \onabla_{e_i} W^{k-1},\lambda^2| W^{k-1}|^{p-4}\langle \onabla_{e_i} W^{k-1}, W^{k-1}\rangle  W^{k-1}\rangle dv_g\\
\leq&~(2-p)\int_M\sum_{i=1}^m|\onabla_{e_i} W^{k-1}|^2\lambda^2| W^{k-1}|^{p-2}dv_g,
\end{aligned}
\end{equation}
where the inequality of $(\ref{2})$ follows from Cauchy-Schwartz inequality.

Substituting $(\ref{1})$ and $(\ref{2})$ into $(\ref{ast})$, we have

\begin{equation*}
\begin{aligned}
&\int_M\sum_{i=1}^m\lambda^2| W^{k-1}|^{p-2}|\onabla_{e_i} W^{k-1}|^2dv_g\\
&\leq~\varepsilon \int_M\sum_{i=1}^m \lambda^2| W^{k-1}|^{p-2} |\onabla_{e_i} W^{k-1}|^2dv_g\\
&\ \ \ \ \ \ \ \ +\frac{1}{\varepsilon}\int_M(\nabla \lambda)^2| W^{k-1}|^{p}dv_g\\
&\ \ \ \ \ \ \ \ +(2-p) \int_M\sum_{i=1}^m|\onabla_{e_i} W^{k-1}|^2\lambda^2| W^{k-1}|^{p-2}dv_g.
\end{aligned}
\end{equation*}
Thus we have

\begin{equation}\label{ast2}
\begin{aligned}
(-\varepsilon-1+p)&\int_M\sum_{i=1}^m\lambda^2| W^{k-1}|^{p-2}|\onabla_{e_i} W^{k-1}|^2dv_g\\
&\leq~\frac{1}{\varepsilon}\int_M(\nabla\lambda)^2| W^{k-1}|^{p}dv_g\\
&\leq~\frac{1}{\varepsilon}\frac{C^2}{r^2}\int_M| W^{k-1}|^{p}dv_g.
\end{aligned}
\end{equation}
Since $(M,g)$ is complete, we tend $r$ to infinity.
 By the assumption $\int_M| W^{k-1}|^{p}dv_g<\infty$,
  the right hand side of $(\ref{ast2})$ goes to zero and the left hand side of $(\ref{ast2})$ goes to 
  $$(-\varepsilon-1+p)\int_M\sum_{i=1}^m| W^{k-1}|^{p-2}|\onabla_{e_i} W^{k-1}|^2dv_g,$$ since $\lambda=1$ on $B_r(x_0)$.
   Since we can take that $\varepsilon>0$ is sufficiently small, by the assumption  $1<p$, we have 
  $$\int_M\sum_{i=1}^m| W^{k-1}|^{p-2}|\onabla_{e_i} W^{k-1}|^2dv_g=0.$$
   From this, we obtain for any vector field $X$ on $M$,
  \begin{equation}\label{eq. for constant}
  \onabla_X W^{k-1}=0.
  \end{equation}
 By $(\ref{eq. for constant})$,
  $$X| W^{k-1}|^2=2\langle \onabla_X W^{k-1}, W^{k-1}\rangle=0.$$
  Therefore we obtain $| W^{k-1}|$ is constant. 
\end{proof}

Before proving our main result, we recall Gaffney's theorem (cf. \cite{mpg1}).

\vspace{5pt}

\begin{thm}
[\cite{mpg1}]\label{Gaffney}
Let $(M,g)$ be a complete Riemannian manifold. 
If a $C^1$ 1-form $\omega$ satisfies that $\int_M|\omega|dv_g<\infty$ and $\int_M(\delta \omega)dv_g<\infty$, or equivalently, a $C^1$ vector field $X$ defined by $\omega (Y)=\langle X,Y\rangle $,\ \ $(\forall ~Y\in \frak{X}(M))$ satisfies that $\int_M|X|dv_g<\infty$ and $\int_M{\rm div}(X) dv_g<\infty$, then
$$\int_M(\delta \omega)dv_g=\int_M{\rm div}(X)dv_g=0.$$
\end{thm}

\vspace{5pt}

By using Lemma $\ref{key lem 1}$ and Theorem $\ref{Gaffney}$, we shall show our main theorem.

\vspace{5pt}

\begin{proof}[Proof of Theorem $\ref{main Th}$]
By Lemma ${\ref{key lem 1}}$, we have $\onabla _{X} W^{k-1}=0$ for any vector field $X$ on $M,$ and $| W^{k-1}|$ is constant.

We shall show the case $(ii)$. If ${\rm Vol}(M,g)=\infty$ and $| W^{k-1}|\not=0$, then
$$\int_M| W^{k-1}|^{p}dv_g=| W^{k-1}|^{p}{\rm Vol}(M,g)=\infty,$$
 which yields the contradiction.
 
We shall show the case $(i)$.
 Define a 1-form $\omega$ on $M$ by
$$\omega(X):=| W^{k-1}|^{\frac{p}{2}-1}\langle \onabla_{X}W^{k-2}, W^{k-1}\rangle,\ \ \ \ \ (X\in \frak{X}(M)).$$
By the assumption $\int_M|\onabla W^{k-2}|^2dv_g<\infty$ and $\int_M| W^{k-1}|^{p}dv_g<\infty$,
 we have

 \begin{equation}\label{Assumption of Gaffney 1}
 \begin{aligned}
 \int_M|\omega|dv_g
 =&\int_M\left(\sum_{i=1}^m|\omega(e_i)|^2\right)^{\frac{1}{2}}dv_g\\
 \leq&~\int_M| W^{k-1}|^{\frac{p}{2}}|\onabla W^{k-2}|dv_g\\
 \leq&\left(\int_M| W^{k-1}|^{p}dv_g\right)^{\frac{1}{2}}\left(\int_M|\onabla W^{k-2}|^2dv_g\right)^{\frac{1}{2}}<\infty.
 \end{aligned}
 \end{equation}
 We consider $-\delta \omega=\displaystyle\sum_{i=1}^m(\nabla_{e_i}\omega)(e_i)$.

 \begin{equation}
 \begin{aligned}
 -\delta \omega
 =&\sum_{i=1}^m\nabla_{e_i}(\omega(e_i))-\omega(\nabla_{e_i}e_i)\\
 =&\sum_{i=1}^m\Big\{\nabla_{e_i}\Big(| W^{k-1}|^{\frac{p}{2}-1}\langle \onabla_{e_i} W^{k-2}, W^{k-1}\rangle\Big)\\
&\hspace{30pt}-| W^{k-1}|^{\frac{p}{2}-1}\langle \onabla_{\nabla_{e_i}e_i} W^{k-2}, W^{k-1}\rangle\Big\}\\
=&\sum_{i=1}^m\Big\{ | W^{k-1}|^{\frac{p}{2}-1}\langle \onabla_{e_i} \onabla_{e_i} W^{k-2},  W^{k-1}\rangle\\
&\hspace{30pt}-| W^{k-1}|^{\frac{p}{2}-1}\langle \onabla_{\nabla_{e_i}e_i} W^{k-2},  W^{k-1}\rangle\Big\}\\
=&\sum_{i=1}^m\Big\{ | W^{k-1}|^{\frac{p}{2}-1}\langle \olapla W^{k-2},  W^{k-1}\rangle\Big\}\\
=&| W^{k-1}|^{\frac{p}{2}+1},
 \end{aligned}
 \end{equation}
where the third equality follows from $| W^{k-1}|$ is constant and $\onabla_X W^{k-1}=0, ~~~(X\in \frak{X}(M)).$
 Since $| W^{k-1}|$ is constant and $\int_M| W^{k-1}|^{p}dv_g<\infty$, the function $-\delta\omega$ is also integrable over $M$. 
 From this and $(\ref{Assumption of Gaffney 1})$, we can apply Gaffney's theorem for the $1$-form
 $\omega.$
  Therefore we have
$$0=\int_M(-\delta \omega)dv_g=\int_M| W^{k-1}|^{\frac{p}{2}+1}dv_g,$$
which implies that $ W^{k-1}=0.$

\end{proof}

\vspace{5pt}

\begin{proof}[Proof of Corollary $\ref{main Cor}$]
By using Theorem $\ref{main Th}$,  we have $W^{k-1}=0$. By repeating this procedure we have $W^{k-2}=0$, $W^{k-3}=0$, $\cdots.$
 Finally, we obtain $\tau(\p)=W^1=0.$
\end{proof}






\quad\\

\bibliographystyle{amsbook}

\end{document}